\documentclass[12pt]{article}

\usepackage{amssymb,amscd,amsthm, verbatim,amsmath,color,fancyhdr, mathrsfs}
\usepackage{enumitem}
\usepackage{soul}
\usepackage{graphicx}
\usepackage[colorinlistoftodos]{todonotes}
\usepackage{comment}

\DeclareMathOperator*{\diam}{diam}

\DeclareMathOperator*{\dom}{dom}
\DeclareMathOperator*{\supp}{supp}

\newtheorem{thm}{Theorem}[section]
\newtheorem{prop}[thm]{Proposition}
\newtheorem{lem}[thm]{Lemma}
\newtheorem{cor}[thm]{Corollary}
\newtheorem{conj}[thm]{Conjecture}
\theoremstyle{definition}

\newtheorem{defn}[thm]{Definition}

\title{Locally Contractive Maps on Perfect Polish Ultrametric Spaces}

\author{Francis George}

\date{\today}

\begin{document}
\maketitle

\begin{abstract}
In this paper we present a result concerning locally contractive maps defined on subsets of perfect Polish ultrametric spaces (i.e. separable complete ultrametric spaces). Specifically, we show that a perfect compact ultrametric space cannot be contained in its locally contractive image, a corollary relating this result to minimal dynamical systems, and pose a conjecture for the general Polish ultrametric case.
\end{abstract}

\section{Preliminaries}

A \emph{topological space} $X$ is \emph{completely metrizable} if $X$ admits a compatible \emph{complete} metric $d$, and \emph{separable} if $X$ contains a countable \emph{dense} subset. A \emph{Polish space} is a separable completely metrizable topological space. Likewise, a metric space that is both separable and complete we will call a \emph{Polish metric space}. Alexander Kechris provides a thorough treatment of Polish spaces in~\cite{Kechris95} section 3 p.13, and we follow the notation contained in this reference for the majority of the paper. 

This paper focuses on a particular subclass of Polish spaces, namely \emph{perfect} Polish \emph{ultrametric} spaces. A topological space is perfect if and only if it contains no \emph{isolated} points.

\begin{defn}\label{ultra}
A metric space $(X,d)$ is an \emph{ultrametric} space if and only if $d$ has the property that for all $x,y\in X$, $d(x,y)\leq \max\{d(x,z),d(y,z)\}$. In this case we call $d$ an ultrametric.
\end{defn}
Briefly, the \emph{diameter} of a set $A$, will be denoted $\diam(A)$, where $\diam(A)=\sup\{ d(x,y): x,y\in A\}$. The following are some fundamental facts about ultrametric spaces which we will state without proof, see
\cite{Schikhof07} p.58, \cite{Kechris95} p.35.

\begin{prop}\label{ultrametricproperties}
For an ultrametric space $(X,d)$ and \emph{ball} $B(x,r)\subseteq X$ with center $x$ and radius $r$, the following holds for all $x,y,z \in X$ and nonempty $Y\subseteq X$,
\end{prop}
\begin{enumerate}[label=(\roman*),noitemsep]
\item $d(x,z)\neq d(y,z)$, implies $d(x,y)=\max\{d(x,z),d(y,z)\}$,
\item All balls are clopen sets, \label{allballsareclopen}
\item If two balls meet, then one is contained in the other,\label{ifballsmeet}
\item If $y\in B(x,r)$, then $B(x,r) = B(y,r)$,
\item If $a\in Y$, then $\diam (Y) = \sup\{d(a,b): b\in Y\}$.\label{radialdiameter}
\end{enumerate}
Recall that a topological space $X$ is \emph{zero-dimensional} if and only if $X$ is \emph{Hausdorff} and has a basis consisting of \emph{clopen} sets. Therefore, any Polish space that admits a compatible ultrametric is zero-dimensional, see \cite{Kechris95} 7.1.

Some examples of zero-dimensional spaces that admit a ultrametric compatible with their usual topology  are the p-adic numbers $\mathbb{Q}_p$, the Cantor Space $2^{\omega}$, the Baire Space $\omega^{\omega}$.

A \emph{Lipschitz} function possesses a stronger form of continuity, which is defined below.

\begin{defn}\label{lipschitz}
For metric space $(X,d)$, a function $f:X\longrightarrow X$ is \emph{Lipschitz} if and only if there is some constant $\alpha \geq 0$ such that $d(f(x),f(y))\leq \alpha d(x,y)$ for all $x,y\in X$. If $0<\alpha < 1$, then $f$ is a \emph{contractive} map.
\end{defn} 

There exist other definitions for contractive functions in the mathematical literature. For the interested reader, a thorough comparison of various definitions for contractive maps was initiated by B.E. Rhoades see \cite{Rhoades77}. But, for the purposes of this paper we have chosen a commonly used definition.
In particular, this definition of a contractive map is fundamental to the \emph{Banach Fixed Point Theorem}, a theorem of Analysis, also known as the \emph{Banach Contraction Principle}~\cite{RoydenFitz10}.

A \emph{neighborhood (nbhd)} of $x$, denoted $N_x$, is an open set containing the point $x$ of a topological space $X$.

\begin{defn}\label{locallycontractive}
Let $(X,d)$ be a metric space. Then a map $f:X\rightarrow X$ is a \emph{locally contractive} map, if and only if for each $x\in X$, there exists a $N_x\subseteq X$ such that for some $\alpha_x\in(0,1)$, $d(f(u),f(v)) \leq \alpha_x d(u,v)$ for all $u,v \in N_x$.
\end{defn}

Notice, the radial nature of the diameter of a ball in an ultrametric space allows us to consider a weaker version of a locally contractive map for the remainder of this paper, see Proposition \ref{ultrametricproperties} item \ref{radialdiameter}. 

\begin{defn}\label{localradialcontraction}
Let $(X,d)$ be a metric space. Then a map $f:X\rightarrow X$ is a \emph{local radial contraction}, if and only if for each $x\in X$, there exists a $N_x\subseteq X$ such that for some $\alpha_x\in(0,1)$, $d(f(x),f(u)) \leq \alpha_x d(x,u)$ for all $u \in N_x$.
\end{defn}

For clarity and convenience, we will introduce the definition of an $f$-\emph{contractive} nbhd.

\begin{defn}\label{fnbhd}
Let $(X,d)$ be a metric space, $f:X\rightarrow X$ a local radial contraction, and let $N_x\subseteq X$ be a nbhd of $x$ for some $x\in X$. Then $N_x$ is an \emph{$f$-contractive} nbhd of $x$, if and only if for some $\alpha_x\in(0,1)$ we have for all $u \in N_x$, $d(f(x),f(u))\leq \alpha_x d(x,u)$.
\end{defn}

Our next lemma demonstrates the effect a local radial contraction $f$ has on compact $f$-contractive nbhds.
\begin{lem}\label{shrinks}
Let $(X,d)$ be a perfect Polish ultrametric space, let the map $f:X\rightarrow X$ be a local radial contraction, and $N_x\subseteq X$ be a compact $f$-contractive nbhd for some $x\in X$. Then, for some $\alpha\in(0,1)$, $\diam(f[N_x])\leq \alpha\diam(N_x)$.
\end{lem}
\begin{proof}
Let $X$ be a perfect Polish ultrametric space, and let the map $f:X\rightarrow X$ be a local radial contraction. Also, let $N_x\subseteq X$ be a compact $f$-contractive nbhd for some $x\in X$. By Definition~\ref{fnbhd}, we know there exists an $\alpha\in (0,1)$ such that for all $u\in N_x$, $d(f(x),f(u)) \leq\alpha d(x,u)$. Since $f[N_x]$ is compact, by the continuity of $f$, it follows that $\diam(N_x)=d(x,x_1)$ for some $x_1\in N_x$ such that $x_1\neq x$, by perfectness, and $\diam(f[N_x])=d(f(x),f(x_2))$ for some $x_2\in N_x$. So, since $ N_x$ is a $f$-contractive nbhd, we have that $d(f(x)),f(x_2)) \leq \alpha d(x,x_2)\leq \alpha d(x,x_1)$. Therefore,
$$
\diam(f[N_x]) \leq \alpha \diam(N_x).
$$
\end{proof}

We will now state a useful fact without proof (see W. H. Schikhof~\cite{Schikhof07} p.48-49).  

\begin{thm}[W. H. Schikhof]\label{pairwise}
Let $(X,d)$ be an ultrametric space. If $U\subseteq X$ is a nonempty open set, then $U$ can be partitioned into balls.
\end{thm}
We now package Theorem \ref{pairwise} into a form more convenient for use in the proofs to follow, and note that compact metric spaces are necessarily Polish.
\begin{lem}\label{partition_into_contractive_balls}
Let $(X,d)$ be a perfect Polish ultrametric space, and let $f:X\rightarrow X$ be a local radial contraction. Then, there exists a finite partition of $X$ into f-contractive balls.
\end{lem}
\begin{proof}(Sketch)
Let $(X,d)$ be a perfect compact ultrametric space, and let $f:X\rightarrow X$ be a local radial contraction. Let $\mathscr{N}_f$ be a collection of $f$-contractive nbhds corresponding to each $x\in X$. This collection forms an open cover of $X$, and by compactness can be assumed to be finite. By Theorem \ref{pairwise} we can create a partition of each  nbhd in $\mathscr{N}_f$ into balls, and again by compactness each of these partitions can be assumed to be finite. Thus, by ultrametric property \ref{ultrametricproperties} \ref{ifballsmeet} it is easy to see that we can form $\mathscr{B}$ a finite partition of $X$ into $f$-contractive balls.
\end{proof}

\section{Main Results}

The proof of our main theorem relies upon the concept of a \emph{generalized tree} or $\emph{R-Tree}$ introduced by Gao and Shao \cite{GaoShao10}. Which generalize the concept of a \emph{descriptive set-theoretic tree} on $\omega$ (see Section 2 Chapter 1 \cite{Kechris95}), by allowing the tree's \emph{level} or \emph{depth} to be indexed by any \emph{countable} set $R$ with limit point $0$ instead of just by $\omega$. This is done through the introduction of a \emph{distance set} for a metric space $(X,d)$, that is the set $R = \{d(x,y): x\neq y \in X\}$.

Distance sets of separable ultrametric spaces are countable, see J. D. Clemens \cite{Clemens07} and  C. Shao \cite{Shao09}.  The set $\mathbb{R}_+$ denotes the set of non-negative real numbers.

\begin{defn}[Gao and Shao \cite{GaoShao10}]
For countable $R\subseteq\mathbb{R}_+$, $(X,d)$ is an {\emph{$R$-ultrametric space}}  if $\{d(x,y): x\neq y \in X\}\subseteq R$.
\end{defn}

For the remainder of this section, unless otherwise specified, fix $R\subseteq \mathbb{R}_+$ to be countable with $0 \in R^\prime$, where $R^\prime$ denotes the set of all limit points of $R$.

\begin{defn}[Gao and Shao \cite{GaoShao10}]
Let $\omega^{<R}$ denote the set of all functions.
$$
u:[a,+\infty)\cap R \rightarrow \omega
$$
where $a\in R$, such that the set
$$
\supp(u)=\{b\in R\cap [a,+\infty): u(b)\neq 0 \}
$$
is finite. We call $(\omega^{<R},\subseteq)$ the {\emph{full $R$-tree}}.

\end{defn}

For some $b\in \dom(u)$, we denote $u\upharpoonright b$ the function $u \upharpoonright ([b,+\infty) \cap R)\in \omega^{<R}.$ 
For $u,v\in\omega^{<R}$, $u$ is an \emph{initial segment} of $v$, denoted $u\subseteq v$, if there exists $b\in \dom(v)$ s.t. $u=v\upharpoonright b$.
For every $u\in\omega^{<R}$ the {\emph{level}} or \emph{depth} of $u$ is defined by $l(u)=\inf\dom(u)=\min\dom(u)$. For $T\subseteq \omega^{<R}$, if $v\in T$ and $u\subseteq v$ implies $u\in T$, then we call $T$ an {\emph{$R$-tree}}, and $T$ is {\emph{pruned}} if for all $u\in T$, there exists a {\emph{proper extension}} $v$ of $u$. A {\emph{branch}} of $T$ is a function $f\in\omega^{R}$ such that for all $a\in R$, $f\upharpoonright a = f\upharpoonright ([a,+\infty) \cap R) \in T$, and $u\subseteq f$ if $u=f\upharpoonright a$. We will let $[T]$ denote the set of all branches, also known as the {\emph{end space}} or {\emph{body}} of $T$, (Gao and Shao \cite{GaoShao10}).

\begin{defn}[Gao and Shao \cite{GaoShao10}]
Let $T$ be an $R$-tree. Define a metric on $[T]$ by
$$
d(f,g) = \left\{ 
\begin{array}{lr}
0, & f=g \\
\max\{a\in R : f(a) \neq g(a) \},  & otherwise
\end{array}
\right.
$$
\end{defn}

\begin{thm}[Gao and Shao \cite{GaoShao10}]
$([T],d)$ is a Polish $R$-ultrametric space.  
\end{thm}
\begin{thm}[Gao and Shao \cite{GaoShao10}]\label{isometric}
For any Polish $R$-ultrametric space $(X,\rho)$ there is a $R$-tree $T$ such that $(X,\rho)$ is isometric to $([T],d)$.
\end{thm}
\begin{defn}[Gao and Shao \cite{GaoShao10}]
For any $u \in \omega^{<R}$ define
$$
N_u=\{f\in[\omega^{<R}]: u\subseteq f\}
$$
\end{defn}

We see above that any branch of an $R$-tree $T$ restricted to any element of of $R$ is an element of $T$.   Together with the fact that for any perfect compact $R$-ultrametric space $X$ the set containing all $u\in T$ where $l(u)=r$ for some $r\in R$ is finite (see p.500 \cite{GaoShao10}), we show the existence of at least one \emph{uniform diameter} finite partition of $X$ into balls, where uniform diameter is meant in the sense that each ball is of equal diameter. Briefly, for any $R$-tree $T$ and $u,v\in T$ if $u\subseteq v$ or $v\subseteq u$ we will call $u$ and $v$ \emph{compatible}, and if this is not the case, \emph{incompatible}, which we will denote by $u \perp v$. Also, it is easy to show that if $u \perp v$, then $N_u \cap N_v = \emptyset$

\begin{lem}\label{uniformcover}
Let $(X,\rho)$ be a perfect compact $R$-ultrametric space. If $R\subseteq \mathbb{R}_+$ is countable and with $R^{\prime} \ni 0$,then $(X,d)$ has a uniform diameter partition for each $r \in R$ such that for any finite cover $\mathscr{B}$ there exists $B_0\in \mathscr{B}$ where $\diam(B_0)=r$ and $\diam(B_0)\leq \diam(B)$ for all $B\in\mathscr{B}$.
\end{lem}

\begin{proof}
Since $R$ is countable and $R^{\prime}\ni 0$ there exists an $R$-tree  $T$ such that exists isometry $\sigma:X \rightarrow [T]$. Clearly, $([T],d)$ is an infinite compact ultrametric space. Therefore, we can refine any cover of $[T]$ to a finite partition $\{N_{u_i}: i < n \}$ where each $u_i \in T$ for some $n<\omega$. Without loss of generality, we can assume this partition is ordered in a non-decreasing fashion with respect to level, and thus label $l_0=l(u_0)$. 

Notice, since the set $\{ u \in T : l(u)=l_0 \}$ is finite it follows that the set $V = \{u,v\in T:  l(u)=l(v)=l_0 \,\,\&\,\, u \perp v\}$ is finite. Then it follows that the collection $\mathscr{C} = \{N_u: u\in V\}$ is a finite partition of $[T]$, and then its easy to see, through use of our isometry $\sigma$, that there exists a finite partition $\mathscr{B}$ with the desired properties.
\end{proof}

The following result by John D. Clemens is helpful in our next theorem. 

\begin{thm}[J. D. Clemens \cite{Clemens07}]\label{enumerateddistance} A set $A\subseteq[0,\infty)$ is a set of distances of some perfect, compact, ultrametric space if and only if A can be enumerated as a countable decreasing sequence $\langle d_i: i\leq 0 \rangle$ with $\lim_{i\rightarrow\infty} d_i = 0$.
\end{thm}

 Now we are ready to present the proof of our main theorem.

\begin{thm}\label{pnotinfp}
Let $(X,\rho)$ be a nonempty perfect Polish ultrametric space. If $X$ is compact, then any local radial contraction $\varphi:X\rightarrow X$ is necessarily not surjective.
\end{thm}
\begin{proof}
Let $(X,\rho)$ be a perfect compact Polish ultrametric space.  Let $R\subseteq\mathbb{R}_+$ such that $R = \{\rho(x,y):x \neq y\in X\}$. Since this is the distance set to a compact perfect ultrametric space, we see by \ref{enumerateddistance} that $R$ can be enumerated as a strictly decreasing sequence $\langle r_j: j\geq 0 \rangle$  and $\lim_{j\rightarrow\infty} r_j = 0$. 

Since $R$ is countable and $R^{\prime}\ni 0$, then by Theorem \ref{isometric}  since $X$ is an Polish $R$-ultrametric space there exists an unique pruned $R$-tree $T$ such that $(X,\rho)$ is isometric to $([T],d)$. So, let $\sigma:X\to [T]$ be an isometry, and define $\tau = \sigma \circ \varphi \circ \sigma^{-1}:[T] \rightarrow [T]$. It is easy to see that $\tau$ is local radial contraction. Thus, it suffices to show that $\tau$ is not surjective.

By  Lemma $\ref{partition_into_contractive_balls}$ we have a finite partition $\{N_{u_i}: i<n\}$ of $\tau$-contractive balls. We may assume that $l(u_i)\leq l(u_j)$ whenever $i\leq j$. Label $l_0=l(u_0)$ (i.e. the lowest level nbhd). Now for each $i<n$, let $\{v^k_i: k<s_i \}$ be the set of all incompatible extensions of $u_i$ with $l(v_i^k)=l_0$, which we know is possible by Lemma~\ref{uniformcover}.

 So, $\diam N_{v^k_i}=r_{t}$ for some $t<\omega$ and $\diam\tau[N_{v^k_i}]\leq r_{t+1}<r_{t}$ for all $i<n$ and $k<s_i$.\newline
Notice for all $i<n$, $\tau[N_{v^k_i}]$ is contained in some ball  $B^k_i$ where $\diam B^k_i \leq r_{t+1}$, and hence by Proposition \ref{ultrametricproperties} (iii) 
$$\tau[N_{v^k_i}]\subseteq B^k_i \subsetneq N_{v^{k^{\prime}}_{i^{\prime}}}$$
 for some $i^{\prime}<n$ and $k^{\prime}<s_{i^{\prime}}$.\newline
 Since, $\{N_{v^k_i}: i<n \,\,\&\,\, k<s_i\}$ is a finite partition of $[T]$ it is easy to see that $\tau[[T]]\subsetneq [T]$.

\end{proof}

Below are some corollaries related to the above result.
The \emph{generic} element of a set has a particular property if and only if the set of all elements with this property forms a \emph{comeager} subset (i.e. a subset which is the complement of a countable union of \emph{nowhere dense} subsets). The \emph{hyperspace} of compact sets denoted $K(X)$ of a topological space $X$ equipped with the 
\emph{Vietoris Topology} which is generated by the sets,
$$
\{K\in X: K\subseteq U\} 
,
\{K\in X: K\cap U \neq \emptyset\},
$$
where $U$ is open in $X$. It is well known that for a perfect Polish space $X$ the hyperspace $K(X)$ is a \emph{Baire} Space (see \cite{Kechris95} p.41), which allows us to take advantage of implications of the \emph{Baire Category Theorem}. So, from \cite{Kechris95} page 42 we see that the set $K_p(X)=\{K\in K(X) : K \text{ is perfect }\}$ is dense $G_\delta$, since $X$ is perfect Polish. Thus, since $K(X)$ is a  Baire space, $K_p(X)$ is comeager, and therefore we see that the generic compact subset of a perfect Polish space is perfect. So, recalling the above theorem, since all self-maps which local radial contractions defined on a perfect compact ultrametric space are not surjective, then the generic compact set of a perfect Polish ultrametric space has this property, which is stated as a corollary below.

\begin{cor}
Let $(X,d)$ be a perfect Polish ultrametric space. Then for generic compact $K\subseteq X$ and local radial contraction $\varphi:K\rightarrow K$, $K\not\subseteq \varphi[K]$.
\end{cor}

Now leading toward a original motivation of the study of this phenomena concerning locally contractive maps, we state the following corollary related to \emph{topological dynamical systems}, which is a pair $\langle X, \varphi \rangle$, where $X$ is a closed topological space, and $\varphi:X\rightarrow X$ is a continuous map. A set $E\subseteq X$ is $\varphi$\emph{-invariant}, if  $\varphi(E)\subseteq E$. A topological dynamical system $\langle X, \varphi \rangle$ is \emph{minimal} if and only if $X$ contains no nonempty closed proper $\varphi$-invariant sets, and in this case $\varphi$ is called a \emph{minimal map}. Observe, that if $\varphi$ is a minimal map and $X$ is compact Hausdorff, then $\varphi$ is surjective (see S. Kolyada and L. Snoha \cite{KolyadaSnoha09}), and from this, the subsequent corollary follows.

\begin{cor}\label{dynamcor}
If $(X,d)$ is a perfect compact ultrametric space and $\langle X,\varphi\rangle$ is a topological dynamical system with $\varphi:X\rightarrow X$ local radial contraction, then $\langle X,\varphi\rangle$ is not minimal.
\end{cor}

In Corollary \ref{dynamcor}, the assumption that $X$ is ultrametric is essential. Krzysztof Ciesielski and Jakub Jasinski recently proved that there exists an uncountable perfect compact subset $X$ of $\mathbb{R}$ and a surjective local radial contraction $f:X\rightarrow X$ such that $\langle X,f \rangle$ is a minimal dynamical system, see \cite{CiesielskiJasinski14}. 

\section{General Polish Ultrametric Case}

Briefly, considering the general Polish ultrametric case as a whole, it should be noted that any finite space (ultrametric or not) is necessarily discrete and thus there exists a trivial locally contractive map (i.e. the identity map). Thus, it suffices for us to consider infinite spaces only. Also, related to the ability to employ an $R$-tree in this characterization, any Polish $R$-ultrametric space such that $R^\prime$ does not contain $0$ is necessarily a countable discrete space (see p.17 and Remark 3.5 p.88 \cite{Shao09}), and thus the identity map again can be used. So, together with the fact that for any Polish ultrametric space $R$ is countable, we can assume that there exists an $R$-tree with end space isometric to the Polish spaces we are considering, and motivated by this, the following conjecture is posed.

\begin{conj}
Let $(X,\rho)$ be an infinite Polish ultrametric space. Then $X$ is compact if and only if any locally contractive map $\varphi:X\rightarrow X$ is necessarily not surjective.
\end{conj}

\section{Acknowledgments}
I would like to thank my faculty advisor Dr. Jakub Jasinksi for his constant support and many helpful suggestion. Our discussions on my paper and topics surrounding my research project are what realized its potential and made me a better mathematician. Also, none of this would be possible if not for my partner Selena, affording me the time necessary to conduct proper research. Part of this research was supported by the University of Scranton President Summer Research Fellowship of 2014.

\end{document}